\documentclass[10pt]{article}

\usepackage{centernot}
\usepackage{amsmath}
\usepackage{amssymb}
\usepackage{amsthm}
\usepackage{hyperref}
\usepackage{url}
\usepackage{graphicx}
\usepackage{subfig}

\newtheorem{theorem}{Theorem}[section]
\newtheorem{lemma}[theorem]{Lemma}

\newtheorem{claim}[theorem]{Claim}

\theoremstyle{definition}
\newtheorem{defn}{Definition}[section]

\theoremstyle{remark}
\newtheorem{remark}{Remark}[section]

\newcommand\Aut{\operatorname{Aut}}

\allowdisplaybreaks

\date{}

\begin{document}

\title{Bernoulli hyper-edge percolation on $\mathbb{Z}^d$}

\author{Yinshan Chang\thanks{Address: College of Mathematics, Sichuan University, Chengdu 610065, China;
Email: ychang@scu.edu.cn;
Supported by National Natural Science Foundation of China \#11701395}}

\maketitle

\begin{abstract}
 We consider Bernoulli hyper-edge percolation on $\mathbb{Z}^d$. This model is a generalization of Bernoulli bond percolation. An edge connects exactly two vertices and a hyper-edge connects more than two vertices. As in the classical Bernoulli bond percolation, we open hyper-edges independently in a homogeneous manner with certain probabilities parameterized by a parameter $u\in[0,1]$. We discuss conditions for non-trivial phase transitions when $u$ varies. We discuss the conditions for the uniqueness of the infinite cluster. Also, we provide conditions under which the Grimmett-Marstrand type theorem holds in the supercritical regime. Such a result is a key for the study of the model in supercritical regime.
\end{abstract}

\section{Introduction}
 We consider Bernoulli hyper-edge percolation on $d$-dimensional integer lattice $\mathbb{Z}^d$ for $d\geq 3$. We start with a description of the model. By a hyper-edge $h$ on $\mathbb{Z}^d$, we mean a \emph{finite} subset of vertices $\{v_1,\ldots,v_m\}$ ($m\geq 2$) of $\mathbb{Z}^d$.  We denote by $\mathcal{H}$ the collection of hyper-edges on $\mathbb{Z}^d$. Let $\Omega=\{0,1\}^{\mathcal{H}}$ be the configuration space. For $\omega\in\Omega$, it is a boolean function on $\mathcal{H}$. We say that the hyper-edge $h$ is open with respect to the configuration $\omega$ if $\omega(h)=1$.  Similarly, it is closed if $\omega(h)=0$. Two vertices $x$ and $y$ are defined to be connected with respect to the configuration $\omega$ if there exist finitely many hyper-edges $h_1,\ldots,h_n$ such that $x\in h_1$, $y\in h_n$, $h_i\cap h_{i+1}\neq\emptyset$ for $i=1,\ldots,n-1$ and $\omega(h_j)=1$ for $j=1,\ldots,n$. In this way, $\mathbb{Z}^d$ is partitioned into finite or countable connected components. These connected components are called ``open clusters''. We assume that for each hyper-edge $h$, there is a continuous and strictly increasing function $p_h:[0,1]\to[0,1]$ such that $p_h(0)=0$ and $p_h(1)=1$. The percolation model has a parameter $u\in[0,1]$. For each parameter $u\in [0,1]$, there is a product probability measure $\mathbb{P}_{u}$ on $\Omega$ such that $\mathbb{P}_u(h\text{ is open})=p_h(u)$. We are particularly interested in the event that there exists an infinite open cluster and its probability under $\mathbb{P}_{u}$.

 Also, we need the model to be translation invariant. To be more precise, for $v\in\mathbb{Z}^d$, we denote by $\tau_v$ the map $\tau_v:\mathbb{Z}^d\to\mathbb{Z}^d$, $w\mapsto v+w$ for $w\in\mathbb{Z}^d$. We use the same notation $\tau_v$ to denote the map
 \begin{equation}\label{eq: lattice shifts}
  \begin{array}{l}
  \tau_v:\mathcal{H}\to \mathcal{H},\\
  \tau_v:\{v_1,\ldots,v_m\}\mapsto\{v+v_1,\ldots,v+v_m\}
  \end{array}
 \end{equation}
 for a hyper-edge $\{v_1,\ldots,v_m\}$.

 We assume throughout this paper that $(p_h)_{h\in \mathcal{H}}$ are translation invariant, i.e.
 \begin{equation}\label{eq: invariance of hyper-edge percolation probabilities}
  p_h(u)=p_{\tau_v h}(u)\text{ for any }v\in\mathbb{Z}^d\text{ and }u\in[0,1].
 \end{equation}

 The assumption is still too general and we make further assumptions on the particular form of $p_{h}(u)$. We assume the open probability $p_{h}(u)$ is associated with a Poisson point process of time parameter $-\ln u$. To be more precise, we assume

 \begin{equation}\label{eq: p defined by mu}
  p_{h}(u)=1-e^{\ln(1-u)\mu(\{h\})}=1-(1-u)^{\mu(\{h\})}.
 \end{equation}

 Here,
  \begin{equation}\label{eq: translation invariance of mu}
  \begin{array}{l}
   \mu\text{ is a }\sigma\text{\emph{-finite measure} on the space of hyper-edges such that}\\
   \mu(\{\tau_{v}h\})=\mu(\{h\})\text{ for all }v\in\mathbb{Z}^d, \text{i.e }\mu\text{ is \emph{translation-invariant}}.
  \end{array}
  \end{equation}

 Many results and their proofs still hold for more general form of $p_h(u)$, see Remark~\ref{rem: generalization} below. However, we decide to assume \eqref{eq: p defined by mu} throughout the paper for simplicity.

 Note that the $\sigma$-finiteness of the measure $\mu$ particularly implies the finiteness of $\mu$ for each hyper-edge, i.e for all hyper-edges $h$,
 \begin{equation}\label{eq: finiteness on one hyper-edge}
  \mu(\{h\})<+\infty.
 \end{equation}

 Furthermore, we want the model has no infinite cluster for small enough $u$. Technically, we need to control the open probability $p_{h}(u)$ for hyper-edges with large diameters: Let $B(n)$ be the box centered at $0$ with side length $2n$, i.e.

  \[B(n)=\{-n,-n+1,\ldots,n-1,n\}^d=([-n,n]\cap\mathbb{Z})^d.\]

  Let $\partial B(n)$ be the inner vertex boundary of $B(n)$, i.e.

  \[\partial B(n)=\{(x_1,x_2,\ldots,x_d)\in\mathbb{Z}^d: \max_{i=1,2,\ldots,d}|x_i|=n\}.\]

  We assume that
  \begin{equation}\label{eq: one hyper-edge annulus crossing general mu}
   \exists \lambda>1\text{ such that }\sup_{n}\mu(\{h:h\cap B(n)\neq\emptyset,h\cap \partial B([\lambda n])\neq\emptyset\})<\infty.
  \end{equation}

  In particular, together with \eqref{eq: finiteness on one hyper-edge}, this implies the local finiteness of $\mu$, i.e.,

  \begin{equation}\label{eq: local finiteness of general mu}
   \mu(\{h:h\text{ contains }0\})<\infty.
  \end{equation}

  This model is a generalization of simple random walk loop percolation model considered in \cite{loop-perc-Zd} with $\mu$ be the so-called loop measure and $1-u=e^{-\alpha}$.

  Firstly, let us describe some first properties of the model (which is known to hold for the loop percolation). The distribution on the configuration space $\Omega=\{0,1\}^{\mathcal{H}}$ is a product measure. Hence, we have the FKG inequality, the (hyper-edge version of) BKR inequality and the Russo's formula. By the same argument as in \cite[Proposition~3.2]{loop-perc-Zd}, we also have the ergodicity of this percolation model under lattice shifts. Also, there is a natural monotone coupling for different parameters $u$.

  Secondly, we discuss the existence on the phase transition. We denote by $0\leftrightarrow\infty$ the event that $0$ is contained in an unbounded open cluster. We define the critical threshold
  \[u_c=\inf\{u\in[0,1]:\mathbb{P}_u[0\leftrightarrow\infty]>0\}.\]
  The existence of phase transition means $u_{c}\in(0,1)$. To be more precise, it means that for sufficiently small $u>0$, $0$ is in a finite open cluster almost surely; for $u$ sufficiently close to $1$, with a strictly positive probability, $0$ is in an unbounded open cluster.

  By the same renormalization argument as in \cite[Lemma~4.1]{loop-perc-Zd}, we have that $u_c>0$. For the other part ``$u_c<1$'', $\mu$ should not be essentially one-dimensional, i.e.
  \begin{multline}\label{eq: not essential one-dimensional}
  \text{there are linearly independent vectors }x,y\in\mathbb{Z}^d\text{ with hyper-edges }\\
  h_1,h_2\text{ such that }\mu(\{h_1\}), \mu(\{h_2\})>0, 0,x\in h_1\text{ and }0,y\in h_2.
  \end{multline}
  The case $h_1=h_2$ is also permitted in the above definition \eqref{eq: not essential one-dimensional}. To see the sufficiency of this condition, we consider the induced finite range dependent bond percolation on the Cayley graph $x\mathbb{Z}+y\mathbb{Z}$ with $x,y$ as the set of generators. We declare an edge $\{z,z+x\}$ on $x\mathbb{Z}+y\mathbb{Z}$ to be open if the hyper-edge $\tau_zh_1$ is open; and we declare an edge $\{z,z+y\}$ to be open if the hyper-edge $\tau_{z}h_{2}$ is open. This bond percolation model has dependency between bonds if $h_1$ is a translation of $h_2$. By \cite{LiggettSchonmannStaceyMR1428500}, for $u_c$ sufficiently close to $1$, this finite range dependent bond percolation model stochastically dominates a supercritical Bernoulli bond percolation on $\mathbb{Z}^2$ from above. Hence, $u_c<1$. To summarize, we have the following result:

  \begin{theorem}\label{thm: existence of phase transition}
  This model exhibits a phase transition (i.e. $u_c\in(0,1)$) under several natural conditions, namely, Poisson nature \eqref{eq: p defined by mu}, $\sigma$-finiteness and translation-invariance of the underlying intensity measure $\mu$ \eqref{eq: translation invariance of mu}, the finiteness of the annulus crossing measures for large enough annuli \eqref{eq: one hyper-edge annulus crossing general mu} and the multi-dimensional assumption \eqref{eq: not essential one-dimensional}.
  \end{theorem}
  \begin{remark}\label{rem: necessarity on the finiteness of the annulus crossing}
   If \eqref{eq: one hyper-edge annulus crossing general mu} is violated, then $u_c=0$ is possible. To be more precise, we denote by $Q_n$ the following square loop
   \[\{(x,y)\in\mathbb{Z}^2:(|x|-2^n)(|y|-2^n)=0\text{ and }|x|,|y|\leq 2^n\}\times\{0\}^{d-2}.\]
   We define $\mu(\{Q_n\})=n\times 2^{-2n}$. We require that $\mu$ is supported on the translations of $Q_n$ ($n\geq 1$). Also, we require that $\mu$ is translation invariant. Then, $\mu$ satisfies \eqref{eq: local finiteness of general mu} but not \eqref{eq: one hyper-edge annulus crossing general mu}. Moreover, in this model, we have the following claim.
   \begin{claim}\label{claim: square loop percolates}
   For all $u>0$, there exist infinite clusters almost surely.
   \end{claim}
   From this example, we see that one need certain restriction on the shape of the hyper-edges in the support of $\mu$ if one wants to weaken the condition \eqref{eq: one hyper-edge annulus crossing general mu}.
  \end{remark}

  In the following, instead of the multi-dimensional assumption \eqref{eq: not essential one-dimensional}, we assume a stronger \emph{irreducibility condition}, which ensures that this percolation model could not be simply reduced to a percolation model on a sub-lattice of $\mathbb{Z}^d$:
  \begin{multline}\label{eq: irreducibility}
   \text{For each pair of vertices }x\text{ and }y\text{, there exists a finite sequence of}\\
   \text{hyper-edges }h_1,h_2,\ldots,h_n\text{ with strictly positive weight under }\mu\\
   \text{ connecting }x\text{ to }y\text{, i.e. }x\in h_1, y\in h_n, \text{for }i=1,2,\ldots,n-1,\\
    h_i\cap h_{i+1}\neq\emptyset\text{ and for }i=1,2,\ldots,n, \mu(\{h_i\})>0.
  \end{multline}

  Thirdly, we discuss the number of infinite (open) clusters. (By infinite open cluster, we mean an unbounded open cluster.) On many non-amenable graphs, the bond percolation have infinitely many infinite clusters for certain range of the parameter. While on amenable graphs like $\mathbb{Z}^{d}$, the infinite cluster must be unique. For our model, we are in the second situation.
  \begin{theorem}\label{thm: uniqueness of the infinite cluster}
   Assume \eqref{eq: p defined by mu}, \eqref{eq: translation invariance of mu} and \eqref{eq: irreducibility}. For $u\in[0,1]$,
   \[\mathbb{P}_{u}[\text{There exist more than two infinite clusters}]=0.\]
  \end{theorem}
  The proof is an adaptation of Burton-Keane argument \cite{BurtonKeaneMR990777}, which is given in Section~\ref{sect: uniqueness of the infinite cluster}.

  Fourthly, we consider the behavior of the model in the supercritical regime, i.e for models with $u>u_c$. For technical reasons, we need to assume the invariance of $\mu$ under symmetry of $\mathbb{Z}^d$:
  \begin{equation}\label{eq: mu invariance under aut Zd}
   \forall \varphi\in \Aut(\mathbb{Z}^d)\text{ and a hyper-edge }h,\quad \mu(\{\varphi(h)\})=\mu(\{h\}),
  \end{equation}
  where $\Aut(\mathbb{Z}^d)$ is the automorphism group of the graph $\mathbb{Z}^d$. Note that the symmetry \eqref{eq: mu invariance under aut Zd} actually implies the irreducibility \eqref{eq: irreducibility}. (A hyper-edge contains at least two different vertices by definition.)

  For the subcritical regime $u<u_c$, the geometric properties of the open clusters depends heavily on the behavior of the measure $\mu$. However, in contrast, the behavior of the clusters in the supercritical regime has less dependence on $\mu$. Roughly speaking, for all such $\mu$, the infinite cluster looks like $\mathbb{Z}^d$ in many perspectives like graph distance, volume growth and so on.

  In the classical bond percolation, such results are proved by using a dynamic renormalization \cite{GrimmettMarstrandMR1068308}. A key result states that the original percolation problem is the limit of its slab versions. To be more precise, we define the slab percolation for our model. Consider a configuration $\omega\in\{0,1\}^{\mathcal{H}}$. We define the truncated configuration $\omega^{\leq L}$ by removing all hyper-edges outside of a two dimensional wide slab.
  \begin{equation}
   \omega^{\leq L}(h)=\left\{\begin{array}{ll}
    \omega(h), & \text{if }h\text{ belongs to }\mathbb{Z}_{\geq 0}^{2}\times \{1,2,\ldots,L\}^{d-2};\\
    0, & \text{otherwise}.
   \end{array}\right.
  \end{equation}

  \begin{theorem}\label{thm: grimmett marstrand}
   We assume $d\geq 3$ and conditions \eqref{eq: p defined by mu}, \eqref{eq: translation invariance of mu}, \eqref{eq: one hyper-edge annulus crossing general mu} and \eqref{eq: mu invariance under aut Zd}. For $u>u_c$, there exists $L=L(u)$ large enough such that there exist infinite clusters with respect to the truncated configuration $\omega^{\leq L}$ almost surely.
  \end{theorem}
  The proof is an adaptation of Grimmett-Marstrand theorem, which is given in Section~\ref{sect: grimmett marstrand}. From this result for the existence of percolation in slabs, we may obtain many properties on finite clusters and the infinite cluster in the supercritical regime, just like in the case of loop percolation \cite{ChangMR3692311}. For example, we have the exponential decay on one-arm connectivity for finite clusters in the supercritical regime.

  \begin{remark}
   In the loop percolation, for $d=2$, Grimmett-Marstrand type theorem may not hold, see \cite[Section~6]{ChangMR3692311}. As the loop percolation is a special case of the hyper-edge percolation, Theorem~\ref{thm: grimmett marstrand} doesn't necessarily hold for $d=2$.
  \end{remark}

  Finally, we would like to remark that the condition \eqref{eq: p defined by mu} could be weakened in various way for various results.
  \begin{remark}\label{rem: generalization}
   Write $p_h(u)=1-(1-u)^{\mu(u,h)}$ for $u\in(0,1)$. Fix $u\in(0,1)$. Assume the translation-invariance and \eqref{eq: irreducibility} for $\mu(u,h)$ instead of $\mu$. (We allow $\mu(u,h)=+\infty$ for some hyper-edge $h$.) Then, we still have
   \[\mathbb{P}_{u}[\text{There exist more than two infinite clusters}]=0.\]
   Assume the translation-invariance and \eqref{eq: not essential one-dimensional} for $\mu(u,h)$ instead of $\mu$. Then, $u_c<1$. Fix some $\varepsilon>0$, define $\mu^{(\varepsilon)}(h)=\sup_{u\in(0,\varepsilon)}\mu(u,h)$. Suppose \eqref{eq: translation invariance of mu} and \eqref{eq: one hyper-edge annulus crossing general mu} hold for $\mu^{(\varepsilon)}$ instead of $\mu$. Then, $u_c>0$.

   All the corresponding proofs are essentially the same.

   For Theorem~\ref{thm: grimmett marstrand}, it suffices to assume that $\forall v>u_c$, $\exists u\in(u_c,v)$ such that
   \begin{equation}\label{eq: gm generalization condition}
    \Delta=\inf_{h\in\mathcal{H}}\frac{\ln(1-p_h(v))}{\ln(1-p_h(u))}-1>0.
   \end{equation}
   Indeed, for hyper-edge percolation models under the assumption \eqref{eq: p defined by mu}, we call it the model of \emph{Poisson type}. For a general model with parameter $v$, under the assumption \eqref{eq: gm generalization condition}, it stochastically dominates the model of Poisson type with the parameter $1-(1-u)^{\Delta+1}>u$ and the intensity measure $\mu(u,h)$. This model of Poisson type is in the supercritical regime and Theorem~\ref{thm: grimmett marstrand} holds. By the stochastic dominance, Theorem~\ref{thm: grimmett marstrand} also holds for the general model with the parameter $v$.
  \end{remark}

\emph{Organization of the paper:} We fix the notation in Section~\ref{sect: notation}. In Section~\ref{sect: uniqueness of the infinite cluster}, we prove Theorem~\ref{thm: uniqueness of the infinite cluster}. In Section~\ref{sect: grimmett marstrand}, we prove Theorem~\ref{thm: grimmett marstrand}.

\section{Definitions and notation}\label{sect: notation}
We fix the notation:
\begin{itemize}
 \item $B(n)=\{-n,-n+1,\ldots,n\}^d$ is the box of side length $2n$ centered at the origin and $B(x,n)=x+B(n)$ is the box of side length $2n$ centered at the vertex $x$ for $x\in\mathbb{Z}^d$ and $n\geq 0$.
 \item $\partial B(n)=\{x\in\mathbb{Z}^d:||x||_{\infty}=n \}$ and $\partial B(x,n)=x+\partial B(n)$ is the inner vertex boundary of $B(x,n)$ for $x\in\mathbb{Z}^d$ and $n\geq 0$.
 \item $F(n)=\{x\in\partial B(n):x_1=n\}$ is a face of the box $B(n)$.
 \item $T(n)=\{x\in\partial B(n):x_1=n,x_j\geq 0\text{ for }j\geq 2\}$.
 \item $T(m,n)=\bigcup\limits_{j=0}^{2m}\{je_1+T(n)\}$, where $e_1=(1,0,\ldots,0)\in\mathbb{Z}^d$.
 \item For a vertex $x\in\mathbb{Z}^d$ and a hyper-edge $h$, we say $x$ is connected to $h$ and write $x\leftrightarrow h$ if there exist hyper-edges $h_1,h_2,\ldots,h_n$ such that $x\in h_1$, $h_i\cap h_{i+1}\neq\emptyset$ for $i=1,2,\ldots,n-1$ and $h_n\cap h\neq\emptyset$.
 \item For a vertex $x\in\mathbb{Z}^d$ and a collection of hyper-edges $\mathcal{H}$, we say $x$ is connected to $\mathcal{H}$ and write $x\leftrightarrow\mathcal{H}$ if there exists $h\in\mathcal{H}$ such that $x\leftrightarrow h$.
 \item For a set $P\subset\mathbb{Z}^d$ of vertices and a collection of hyper-edges $\mathcal{H}$, we say $P$ is connected to $\mathcal{H}$ and write $P\leftrightarrow \mathcal{H}$ if $\exists x\in P$ and $h\in\mathcal{H}$ such that $x\leftrightarrow h$.
\end{itemize}

\section{Uniqueness of the infinite cluster}\label{sect: uniqueness of the infinite cluster}
 The proof is an adaptation of Burton-Keane argument \cite{BurtonKeaneMR990777}.
 Note that we have the following properties:
  \begin{enumerate}
   \item[(Pa)] The model is translation-invariant and ergodic under lattice shifts, which particularly implies that the number of infinite clusters is almost surely a constant. In particular, the probability of the existence of infinite open clusters is either $0$ or $1$.
   \item[(Pb)] If the probability of the existence of infinite open clusters is strictly positive for some $u\in[0,1]$, then \[\lim_{m\rightarrow\infty}\mathbb{P}[B(m)\leftrightarrow\infty]=1,\]
       where $B(m)\leftrightarrow\infty$ means $B(m)$ has non-empty intersection with an infinite (unbounded) open cluster. Moreover, if there exist infinitely many clusters with probability $1$, then for all $k\geq 1$,
       \[\lim_{n\rightarrow\infty}\mathbb{P}[B(n)\text{ intersects at least }k\text{ disjoint infinite clusters}]=1.\]
   \item[(Pc)] By \eqref{eq: translation invariance of mu} and \eqref{eq: irreducibility} for $x=0$ and $y=\pm e_j$ for $j=1,2,\ldots,d$, there exists $K=K(\mu)<\infty$ such that for all finite connected vertex sets $F$ of $\mathbb{Z}^d$, there exists a finite collection $(h_i)_{i\in I}$ of hyper-edges such that $\min_{i\in I}\mu(\{h_i\})>0$, such that for all $x,y\in F$, they are connected by a sequence of hyper-edges in $(h_i)_{i\in I}$. Besides, all $h_i$ ($i\in I$) are within distance $K$ from $F$. Furthermore, for any two different hyper-edges in the collection $(h_i)_{i\in I}$, they are connected by a sequence of hyper-edges in $(h_i)_{i\in I}$. We denote the finite collection of hyper-edges by $\mathcal{H}(F,K)$.
   \item[(Pd)] Suppose that $\mathcal{H}_{1}$ is a finite set of hyper-edges with strictly positive weights under $\mu$. Let $c\in\{0,1\}^{\mathcal{H}_{1}}$ be a fixed configuration on the finite set of hyper-edges $\mathcal{H}_{1}$. Let $1(\mathcal{H}_1)$ be the configuration $(1,1,\ldots,1)$, i.e. it requires each hyper-edges in $\mathcal{H}_{1}$ to be open. For a configuration $\omega\in\{0,1\}^{\mathcal{H}}$ on the whole space, let $\omega|_{\mathcal{H}_{1}}$ be its restriction on $\mathcal{H}_{1}$. Then, for $u>0$, we have that
       \[\mathbb{P}_{u}(\omega|_{\mathcal{H}_{1}}=1(\mathcal{H}_{1}))/\mathbb{P}(\omega|_{\mathcal{H}_{1}}=c)>0\text{ if }P(\omega|_{\mathcal{H}_{1}}=c)>0.\]
  \end{enumerate}
  We fix some $u\in(0,1)$. (The cases $u=0$ and $u=1$ are trivial.) By (Pa), we see that the number of infinite clusters is almost surely a constant $k$. Firstly, we show that $k$ must be $0$, $1$ or $\infty$. Assume that the number of infinite clusters almost surely equals to a finite $k\geq 1$. Then, for large enough $n$, with positive probability, the box $B(n)$ intersects with all the infinite clusters. We denote this event by $F$. Using (Pc), there exists a finite set of hyper-edges $\mathcal{H}(B(n),K)$ with positive weights under $\mu$ such that for each pair of vertices $x$ and $y$ in $B(n)$, they are connected via the hyperedges in $\mathcal{H}(B(n),K)$. Denote by $E$ the event that the edges in $\mathcal{H}(B(n),K)$ are all open. Then, by (Pd), by the independence between the status of the hyper-edges in $\mathcal{H}(B(n),K)$ and those outside of $\mathcal{H}(B(n),K)$ and by the definition of $F$, we also have $P(EF)>0$. Note that $EF$ implies that there exists a unique infinite cluster. By ergodicity, the number of infinite clusters must equal to $1$ almost surely.

  Next, we need to rule out the possibility of infinitely many infinite clusters. We will prove it by contradiction. Assume that there exist infinite many infinite clusters with probability $1$. In the original argument of Burton and Keane \cite{BurtonKeaneMR990777}, they considered trifurcation points. To adapt to our models, we use multifurcation boxes instead of trifurcation points. By multifurcation box of side length $2L$, we mean a box $B=B(x,L)$ with the following properties:
  \begin{enumerate}
    \item[(1)] The hyper-edges in the set $\mathcal{H}(B(x,L-K),K)$ are all open, where $K=K(\mu)$ and $\mathcal{H}(F,K)$ with $F=B(x,L-K)$ are the same as in the property (Pc) at the beginning of the present section.
    \item[(2)] The box $B(x,L-K)$ intersects with some infinite cluster $C$ and by declaring all hyper-edges in $\mathcal{H}(B(x,L-K),K)$ to be closed, the infinite cluster $C$ breaks into more than three disjoint infinite clusters.
  \end{enumerate}
  We claim that there exists $L<\infty$ such that
  \[\mathbb{P}[B(L)\text{ is a multifurcation box}]=p(L)>0.\]
  Indeed, we take $K=K(\mu)$ in (Pc). By taking a large enough $L$,
  \[\mathbb{P}[B(L-K)\text{ intersects with more than three disjoint infinite clusters}]>0.\]
  Denote by $G$ the event that $B(L-K)$ intersects with more than three disjoint infinite clusters. Assume $G$ happens. If we change all the hyperedges in $\mathcal{H}(B(L-K),K)$ to be open, then these disjoint infinite clusters merge into one infinite cluster $C$. Afterwards, if we change all the edges in $\mathcal{H}(B(L-K),K)$ to be closed, then $C$ breaks into more than three infinite clusters intersecting $B(L)$. Hence, by (Pd),
  \[\mathbb{P}[B(L)\text{ is a multifurcation box}]=p(L)>0.\]

  Take a large enough natural number $n$ such that $\sharp\partial B(4nL)\leq \frac{1}{2}p(L)n^d$ and consider the following disjoint boxes inside $B(4nL)$:
  \begin{equation}
   \mathcal{B}_n=\{B(3L\cdot x,L):x\in [-n,n]^d\cap\mathbb{Z}^d\}.
  \end{equation}

  For a multifurcation box $B(3L\cdot x,L)\in\mathcal{B}_n$ inside $B(4nL)$, define its \emph{offspring}
  \begin{equation*}
   Y_{x}=\{x\in\partial B(4nL):x\leftrightarrow \mathcal{H}(B(x,L-K),K)\}.
  \end{equation*}
  We have the following claim:
  \begin{claim}\label{claim: disjoint or identical}
   For different multifurcation boxes $B(3L\cdot x,L)$ and $B(3L\cdot y,L)$, their offsprings either coincide or be disjoint, i.e. $Y_{x}=Y_{y}$ or $Y_{x}\cap Y_{y}=\emptyset$.
  \end{claim}
  \begin{proof}[Proof Claim~\ref{claim: disjoint or identical}]
   Assume $Y_{x}\cap Y_{y}\neq\emptyset$. Suppose $z\in Y_{x}\cap Y_{y}$. Then, there exist two hyper-edges $h_{1}\in \mathcal{H}(B(3Lx,L-K),K)$ and $h_2\in\mathcal{H}(B(3Ly,L-K),K)$ such that $z$ is connected to both $h_1$ and $h_2$. Hence, $h_1$ is connected to $h_2$. Consider $u\in Y_{x}$. Then, there exists $h_3\in\mathcal{H}(B(3Lx,L-K),K)$ such that $u$ is connected to $h_3$. Either $h_1=h_3$ or $h_1$ is connected to $h_3$ within $\mathcal{H}(B(3Lx,L-K),K)$ by the property (Pc). Therefore, $h_3$ is also connected to $h_2$. But $u$ is connected to $h_3$. So, the vertex $u$ is connected to $h_2$. Hence, by definition, we have that $u\in Y_{y}$. Consequently, $Y_{x}\subset Y_{y}$. Similarly, $Y_{y}\subset Y_{x}$ and we must have $Y_{x}=Y_{y}$.
  \end{proof}
  Next, for a multifurcation box $B(3L\cdot x,L)\in\mathcal{B}_n$ inside $B(4nL)$, we define a partition of $\{P_{x,1},P_{x,2},\ldots,P_{x,r(x)}\}$ of $Y_{x}$ as follows: Two vertices $u,v\in Y_{x}$ are said to be in the same cluster if they can still be connected via a sequence of open hyper-edges after we declare all the hyper-edges in $\mathcal{H}(B(3Lx,L-K),K)$ to be closed. Then, by definition of multifurcation boxes, we have $r(x)\geq 3$. For two different multifurcation boxes $B(3L\cdot x,L)$ and $B(3L\cdot y,L)$, if $Y_{x}=Y_{y}$, the corresponding partitions are \emph{compatible} in the following sense.

  Let $Y$ be a finite set with at least three elements. A multi-partition $\Pi=\{P_1,P_2,\ldots,P_{r}\}$ is a partition of $Y$ with $r\geq 3$. Given two multi-partitions $\Pi^{(1)}=\{P^{(1)}_1,P^{(1)}_2,\ldots,P^{(1)}_{r_1}\}$ and $\Pi^{(2)}=\{P^{(2)}_1,P^{(2)}_2,\ldots,P^{(2)}_{r_2}\}$ of $Y$, we say that they are \emph{compatible} if there exists an ordering of their elements such that $P^{(1)}_1\supset P^{(2)}_{2}\cup\ldots\cup P^{(2)}_{r_2}$. (In this case, we also have $P^{(2)}_1\supset P^{(1)}_{2}\cup\ldots\cup P^{(1)}_{r_1}$. Or equivalently, $P^{(1)}_1\cup P^{(2)}_1=Y$.)

  Then, we have the following claim.
  \begin{claim}\label{claim: compatible}
   For two different multifurcation boxes $B(3L\cdot x,L)$ and $B(3L\cdot y,L)$, if $Y_{x}=Y_{y}$, the corresponding partitions $\{P_{x,i}\}_{i=1,2,\ldots,r(x)}$ and $\{P_{y,j}\}_{j=1,2,\ldots,r(y)}$ are compatible.
  \end{claim}
  \begin{proof}[Proof of Claim~\ref{claim: compatible}]
   Firstly, by declaring the hyper-edges in the union of hyper-edges $\mathcal{H}(B(3Lx,L-K),K)\cup\mathcal{H}(B(3Ly,L-K),K)$ to be closed, the set $Y_x=Y_y$ splits into several open clusters. (Two vertices are said to be in the same cluster if they can still be connected via a sequence of remainder open hyper-edges.) Some clusters $P_{x\& y,1},P_{x\& y,2},\ldots,P_{x\& y,r_{0}}$ connect to both $\mathcal{H}(B(3Lx,L-K),K)$ and $\mathcal{H}(B(3Ly,L-K),K)$. Some clusters $P_{x\backslash y,1},P_{x\backslash y,2},\ldots,P_{x\backslash y,r_1}$ connect to $\mathcal{H}(B(3Lx,L-K),K)$ but not to the collection $\mathcal{H}(B(3Ly,L-K),K)$. The remaining clusters $P_{y\backslash x,1},P_{y\backslash x,2},\ldots,P_{y\backslash x,r_2}$ connect to $\mathcal{H}(B(3Ly,L-K),K)$ but not to $\mathcal{H}(B(3Lx,L-K),K)$. Note that $Y_x=Y_y$ is exactly the disjoint union of $\{P_{x\& y,i}\}_{i=1,2,\ldots,r_0}$, $\{P_{x\backslash y,j}\}_{j=1,2,\ldots,r_1}$ and $\{P_{y\backslash x,k}\}_{k=1,2,\ldots,r_2}$. Now, if we declare $\mathcal{H}(B(3Lx,L-K),K)$ back to be open, then,
   \[P_{x\& y,1},P_{x\& y,2},\ldots,P_{x\& y,r_{0}},P_{x\backslash y,1},P_{x\backslash y,2},\ldots,P_{x\backslash y,r_1}\]
   will merge into a single cluster $P_{x,1}$. Similarly, if we keep $\mathcal{H}(B(3Lx,L-K),K)$ closed but declare $\mathcal{H}(B(3Ly,L-K),K)$ back to be open, then, the clusters
   \[P_{x\& y,1},P_{x\& y,2},\ldots,P_{x\& y,r_{0}},P_{y\backslash x,1},P_{y\backslash x,2},\ldots,P_{y\backslash x,r_2}\]
   will merge into a single cluster $P_{y,1}$. Then, we have that $P_{x,1}\cup P_{y,1}=Y_{x}=Y_{y}$ and hence, the corresponding partitions $\{P_{x,i}\}_{i=1,2,\ldots,r(x)}$ and $\{P_{y,j}\}_{j=1,2,\ldots,r(y)}$ are compatible.
  \end{proof}

  We will need to relate the size of $Y_{x}$ and the number multifurcation boxes with the offspring $Y_{x}$. For that reason, we need the following lemma which is an adaptation of \cite[Lemma~8.5]{GrimmettMR1707339}.
  \begin{lemma}\label{lem: compatible multi-partitions}
   Let $\mathcal{P}$ be a compatible family of distinct multi-partitions of $Y$, then $\sharp Y\geq \sharp \mathcal{P}+2$.
  \end{lemma}
  \begin{proof}[Proof of Lemma~\ref{lem: compatible multi-partitions}]
   We prove this by induction on $\sharp Y$. If $\sharp Y=3$, then $\sharp \mathcal{P}\leq 1$ and we have $\sharp Y\geq \sharp\mathcal{P}+2$. Assume that the claim holds if $\sharp Y\leq n$ ($n\geq 3$), and let $Y$ satisfy $\sharp Y=n+1$.

   Pick an element $y\in Y$ and define $Z=Y\setminus\{y\}$. For each multi-partition $\Pi$ of $Y$, it may be expressed as $\Pi=\{Q_1\cup\{y\},Q_2,\ldots,Q_{r}\}$. There are two cases: $Q_1=\emptyset$ or $Q_1\neq\emptyset$. When $Q_1$ is not empty, we have a multi-partition $\Pi^{Z}=\{Q_1,Q_2,\ldots,Q_r\}$ of $Z$. Let $\mathcal{Q}=\{\Pi^{Z}:\{y\}\notin\Pi,\Pi\in\mathcal{P}\}$ be the collection of multi-partitions of $Z$ obtained from the multi-partitions of $Y$ by removing $y$. Then, $\mathcal{Q}$ is compatible family of multi-partitions. Besides, $\sharp \mathcal{P}\leq \sharp Q+1$. Indeed, there exists at most one multi-partition $\Pi$ such that $\{y\}\in\Pi$ by definition of compatibility of multi-partitions.

   By the induction hypothesis, $\sharp Z\geq \sharp \mathcal{Q}+2$. Hence, we have that $\sharp Y=\sharp Z+1\geq \sharp \mathcal{Q}+3\geq \sharp\mathcal{P}+2$.
  \end{proof}

  By Claim~\ref{claim: disjoint or identical}, Claim~\ref{claim: compatible}, Lemma~\ref{lem: compatible multi-partitions}, the total size of those different offsprings is not less than the total number $N$ of multifurcation boxes. However, clearly, the total size of those different offsprings is no more than the size of the boundary $\partial B(4nL)$.

  Note that $\mathbb{E}(N)=p(L)n^d$, where $N$ is the number of multifurcation boxes in $\mathcal{B}_n$. Moreover, since $\sharp \mathcal{B}_n=n^d$, we have $\mathbb{E}(N^2)\leq n^{2d}$. So, by Paley-Zygmund inequality, the number of multifurcation boxes $N$ in $\mathcal{B}_n$ is strictly bigger than $\frac{1}{2}p(L)n^{d}$ with strictly positive probability. Hence, with strictly positive probability,
  \[\sharp\partial B(4nL)\geq N>\frac{1}{2}p(L)n^d.\]
  And for large but fixed $L$, there is a contradiction by amenability of $\mathbb{Z}^d$. From this contradiction, we see that the number of infinite clusters could not be infinite almost surely. Finally, we prove the uniqueness of the infinite cluster.

  \begin{remark}
   It might be possible to prove the result by using trifurcation boxes. However, we feel that it is difficult to achieve or at least difficult to write. One main challenge is a way to glue more than $3$ infinite clusters into exactly $3$ clusters by adding several hyper-edges. (It might happen that certain hyper-edge will glue all these clusters into a single cluster.) Therefore, we decide to use multifurcation boxes instead.
  \end{remark}
  
\section{Grimmett-Marstrand type theorem}\label{sect: grimmett marstrand}

 We will closely follow the proof strategy in \cite[Section~5]{ChangMR3692311} for the loop percolation, which is also an adaptation of the original argument of Grimmett and Marstrand in \cite{GrimmettMarstrandMR1068308}. Roughly speaking, the idea is to use dynamic renormalization. At beginning, one needs to define good and bad boxes. The structure of the boxes almost looks like $\mathbb{Z}_{\geq 0}^2$. By considering good boxes as open vertices, those good boxes form open clusters. Then, one needs to show that this box percolation model has infinite clusters, which by definition of good boxes, ensures an infinite cluster in the slab in the original percolation model. To prove the percolation of good boxes, one need to show that the box percolation model stochastically dominates a supercritical site percolation model on $\mathbb{Z}_{\geq 0}^{2}$. For that purpose, two main properties are important. Firstly, the status of the boxes (i.e. good or bad) is somehow ``locally dependent''. Secondly, the size of the boxes are chosen to be large enough such that a box is good with a probability sufficiently close to $1$. A technical point is that the boxes are explored in certain order and the locations of the boxes are random. In contrast, in static renormalization, the positions of the boxes are fixed and periodic. This explains the name \emph{dynamic} renormalization.

 The main reason for adapting the proof strategy of the loop percolation is the following: They are both percolation models driven by Poisson point processes. To be more precise, consider a Poisson point process on the space of hyper-edges with the intensity measure $\mu$, which is viewed as countable collection of pairs $\{(t_i,h_i)\}_{i}$. Each $h_i$ is a hyper-edge and $t_i$ is the time of appearance of $h_i$. Let
 \begin{equation}\label{eq: G alpha}
  G_{\alpha}=\{h_i:t_i\leq \alpha\}
 \end{equation}
 be the set of hyper-edges that we have collected up to time $\alpha$. Here, $G_{\alpha}$ is regarded as a multi-set. If a hyper-edge $h$ appears several times, we also count the multiplicity. For a hyper-edge $h$, we define it to be open at time $\alpha$ if $h$ belongs to $G_{\alpha}$. By definition of Poisson point process, each hyper-edge, independently from the other hyper-edges, is open with probability $p_h(u)=1-(1-u)^{\mu(\{h\})}$ where $u=1-e^{-\alpha}$. In this way, our percolation model is driven by a Poisson point process and there is a natural coupling for all parameters $u$.

 Due to the similarity between these two models, the rest of the proof basically follows that of loop percolation in \cite[Section~5]{ChangMR3692311} and we will omit many details. However, note that our model is more general. Hence, certain care and modification are necessary, which will be explained as follows:

 Our first step is to modify the definition of a seed event given in \cite[Def. 5.1]{ChangMR3692311}. For that purpose, for a positive function $\beta$ on the space of hyper-edges, similar to \eqref{eq: G alpha}, we define $G_{\beta}$ as $\{h_i:t_i\leq\beta(h_i)\}$. Here, $G$ plays the same role as $\mathcal{L}$ in \cite{ChangMR3692311}.

 Comparing with the definition of a seed event for the loop percolation, we can't require that each pair of vertices in a box are connected with a strictly positive probability. However, by \eqref{eq: translation invariance of mu} and \eqref{eq: irreducibility} for $x=0$ and $y=\pm e_j$ for $j=1,2,\ldots,d$, note that there exists $c=c(\mu)<\infty$ such that the subgraph $B(x,m,c)$ of $B(x,m)$ are totally connected by $\mathcal{G}_{\alpha}$ with positive probabilities for $\alpha>0$, where $B(x,m,c)$ is obtained from $B(x,m)$ by removing $2^d$ many boxes of side length $c$ at those corners of $B(x,m)$. To be more precise, we define
  \begin{equation}
   B(x,m,c)=\left\{y\in B(x,m):\begin{array}{l}\exists i=1,\ldots,d\text{ such that the}\\
   i\text{-th coordinate of }y-x\\
                                    \text{ is contained in }[-(m-c),m-c]\end{array}\right\}.
  \end{equation}

  Accordingly, we define our version of seeds.
  \begin{defn}[Seed event]\label{defn: seed event for hyper-edge percolation}
   Let $m\geq 1$, we call a modified box $B(x,m,c)$ a $\beta$-\emph{seed} if $\forall h\subset B(x,m)$ with $\mu(\{h\})>0$, we have that $h\in\mathcal{G}_{\beta}$. We set
  \begin{equation*}
   K(m,n,\beta)=\{\text{union of the }\beta\text{-seeds lying within }T(m,n)\}.
  \end{equation*}
  When $\beta\equiv\alpha$ is a constant function, we write $\alpha$-seed and $K(m,n,\alpha)$.
  \end{defn}

  The following crucial lemma is analogous to \cite[Lemma~5.1]{ChangMR3692311}.
  \begin{lemma}\label{lem: sprinkling for hyper-edge percolation}
  Suppose $d\geq 3$ and $\alpha>\alpha_c=-\ln(1-u_{c})$. For $\varepsilon,\delta>0$ and an intensity function $\gamma:\mathcal{H}\rightarrow[\alpha,A]$, there exist integers $m=m(d,\alpha,A,\varepsilon,\delta)$ and $n=n(d,\alpha,A,\varepsilon,\delta)$ such that $2m<n$ and the following property holds. Let $R$ be such that $B(m)\subset R\subset B(\lfloor\lambda n\rfloor-1)$. Define $\tilde{\gamma}:\mathcal{H}\rightarrow\mathbb{R}_{+}$ as follows: for a hyper-edge $h$,
  \[\tilde{\gamma}(h)=\left\{\begin{array}{ll}
  \delta & \text{ if }h\subset B(\lfloor\lambda n\rfloor-1),h\cap \partial B(n-1)\neq\emptyset,h\cap R\neq \emptyset,\\
  \gamma(h) & \text{ otherwise.}
   \end{array}\right.\]
  Define the set of hyper-edges $(\mathcal{G}_{\tilde{\gamma}})_{B(n-1)}^{B(\lfloor \lambda n\rfloor-1)}$ by
  \begin{equation}
   (\mathcal{G}_{\tilde{\gamma}})_{B(n-1)}^{B(\lfloor \lambda n\rfloor-1)}=\{h\in \mathcal{G}_{\tilde{\gamma}}:h\cap B(n-1)\neq\emptyset, h\subset B(\lfloor \lambda n\rfloor-1)\}.
  \end{equation}
  Then, we have that
  \[\mathbb{P}\left[R\overset{(\mathcal{G}_{\tilde{\gamma}})_{B(n-1)}^{B(\lfloor \lambda n\rfloor-1)}}{\centernot\longleftrightarrow}K(m,n,\tilde{\gamma}),R\cap K(m,n,\tilde{\gamma})=\emptyset\right]<\varepsilon,\]
  where $R\overset{(\mathcal{G}_{\tilde{\gamma}})_{B(n-1)}^{B(\lfloor \lambda n\rfloor-1)}}{\centernot\longleftrightarrow}K(m,n,\tilde{\gamma})$ means that we cannot connect $R$ to the set $K(m,n,\tilde{\gamma})$ by only using the hyper-edges in $(\mathcal{G}_{\tilde{\gamma}})_{B(n-1)}^{B(\lfloor \lambda n\rfloor-1)}$.
 \end{lemma}
 Note that $K(m,n,\tilde{\gamma})=K(m,n,\gamma)$.

 Once this lemma is verified, the rest of the proof is the same as that of loop percolation. Why do we need this lemma? In the dynamic renormalization, we need to connect a seed to another seed via certain exploration process. To establish a stochastic domination with site percolation, we need to show that a new seed could be connected with a conditional probability sufficiently close to $1$ given the history of the exploration. There could be failure in discovering new seeds. This negative information prevents us from using FKG inequalities. Instead, the idea is to use sprinkling, i.e. we locally add more hyper-edges by slightly changing the time parameter $\gamma$. The connection to a new seed might be unsuccessful. However, after a local and very small increase in the parameter, the goal of connection to a new seed could be achieved with sufficiently high probabilities. This is precisely the value of Lemma~\ref{lem: sprinkling for hyper-edge percolation}. Here, $R$ is the intersection of $B(\lfloor \lambda n\rfloor-1)$ with the union of hyper-edges that we have explored. There are still many unexplored hyper-edges, their status are not determined and they are independent of the exploration history. For example, the hyper-edges related with the definition of $K(m,n,\gamma)$ are such unexplored hyper-edges. The same is true for the hyper-edges disjoint from $R$ and a hyper-edge $h$ intersecting $R$ but appears in the time $(\gamma(h),\gamma(h)+\delta)$. These three kind of hyper-edges is sufficient to find new seed with sufficiently high probabilities.

 To obtain Lemma~\ref{lem: sprinkling for hyper-edge percolation}, we prove several preparation lemmas in sequence.

 Firstly, in the loop percolation on $\mathbb{Z}^d$, we need to show that the intensity measure of large loops are small, see \cite[Lemma~5.2]{ChangMR3692311}. For the hyper-edge percolation, by using the condition \eqref{eq: one hyper-edge annulus crossing general mu}, we control the intensity measure of large hyper-edges via the following lemma.
 \begin{lemma}\label{lem: control of large hyper-edges}
 There exists a non-decreasing function $f:\mathbb{Z}_{+}\to\mathbb{Z}_{+}$ such that $f(n)<n$, that $\lim_{n\to\infty}f(n)=+\infty$ and that
 \begin{equation*}
  \lim_{n\to\infty}\mu(\{h\in\mathcal{H}:h\cap B(f(n))\neq\emptyset, h\cap\partial B(n)\neq\emptyset\})=0.
 \end{equation*}
 \end{lemma}

 \begin{proof}[Proof of Lemma~\ref{lem: control of large hyper-edges}]
  By \eqref{eq: one hyper-edge annulus crossing general mu} and the continuity of the measure $\mu$ from above, for any positive integer $n$, there exists another positive integer $g(n)>n$ such that
  \begin{equation*}
   \mu(\{h\in\mathcal{H}:h\cap B(n)\neq\emptyset, h\cap\partial B(g(n))\neq\emptyset\})\leq 1/n.
  \end{equation*}
  Consider the sequence of larger integers
  \[g(1),\max(g(1),g(2))+1,\max(\max(g(1),g(2))+1,g(3))+1,\ldots\]
  if necessary, we may assume that $n\mapsto g(n)$ is strictly increasing. Take $f$ to be the inverse of $g$. Then, $f$ has the required properties.
 \end{proof}
 In the following, we assume that $f$ satisfies the properties in Lemma~\ref{lem: control of large hyper-edges} and the constant $\lambda>1$ satisfies the condition \eqref{eq: one hyper-edge annulus crossing general mu}. Here, $f(n)$ (resp. $\lambda n$) plays the same role for our model as that of $\sqrt{n}$ (resp. $2n$) in the proof of Grimmett-Marstrand's theorem for the loop percolation.

 Secondly, by using the argument similar to proof of \cite[Lemma~5.3]{ChangMR3692311}, we obtain the following replacement \cite[Lemma~5.3]{ChangMR3692311} of from Lemma~\ref{lem: control of large hyper-edges}:
  \begin{lemma}\label{lem: truncated crossing general mu}
   For fixed $m\geq 1$ and $u=1-e^{-\alpha}$,
  \begin{equation*}
   \lim\limits_{n\rightarrow\infty}\mathbb{P}_{u}[B(m)\longleftrightarrow\partial B(f(n)),B(m)\overset{B(n-1)}{\centernot\longleftrightarrow}\partial B(f(n))]=0,
  \end{equation*}
  where $B(m)\overset{B(n-1)}{\centernot\longleftrightarrow}\partial B(f(n))$ means that there is no connection between $B(m)$ and $\partial B(f(n))$ if we only use the open hyper-edges in $B(n-1)$.
  \end{lemma}

 Thirdly, we obtain the following analogy of \cite[Lemma~5.4]{ChangMR3692311} by following the same proof.
 \begin{lemma}\label{lem: many truncated hyper-edges}
 Fix $m\geq 1$ and a sufficiently large $n$. Define $\mathcal{C}(m,n,\alpha)$ to be the cluster of vertices which can be connected to $B(m)$ by $\{h\in\mathcal{G}_{\alpha}:h\subset B(n-1)\}$. Define
 \[\mu(m,n,\alpha)=\mu(\{h:h\subset B(\lfloor \lambda n\rfloor-1),h\cap\partial B(n)\neq\emptyset,h\cap\mathcal{C}(m,n,\alpha)\neq\emptyset\}).\]
 Let $(\mathcal{G}_{\alpha})^{B(n-1)}=\{h\in\mathcal{G}_{\alpha}:h\subset B(n-1)\}$. Then, for fixed $m,k\geq 1$ and $\alpha>0$,
 \begin{equation*}
  \lim\limits_{n\rightarrow\infty}\mathbb{P}[B(m)\overset{(\mathcal{G}_{\alpha})^{B(n-1)}}\longleftrightarrow \partial B(f(n)),\mu(m,n,\alpha)\leq k]=0,
 \end{equation*}
 where $B(m)\overset{(\mathcal{G}_{\alpha})^{B(n-1)}}\longleftrightarrow \partial B(f(n))$ means that $B(m)$ and $B(f(n))$ are connected via hyper-edges in $(\mathcal{G}_{\alpha})^{B(n-1)}$.
 \end{lemma}

 Fourthly, from Lemma~\ref{lem: many truncated hyper-edges}, we could obtain the following analogy of \cite[Lemma~5.5]{ChangMR3692311}.
 \begin{lemma}\label{lem: many truncated connections to the outer box}
 For $\alpha>0$, we define $U(m,n,\alpha)$ be the set of vertices on $\partial B(n)$ which could be connected to $B(m)$ via
 \[\{h\in\mathcal{G}_{\alpha}:h\cap B(n-1)\neq\emptyset,h\subset B(\lfloor\lambda n\rfloor-1)\}.\]
 Then, for all fixed $k\geq 1$,
 \begin{equation}\label{eq: mtcttob0}
 \lim\limits_{n\rightarrow\infty}\mathbb{P}[\sharp U(m,n,\alpha)\leq k,B(m)\overset{\mathcal{G}_{\alpha}}{\longleftrightarrow} \partial B(f(n))]=0,
 \end{equation}
 where $B(m)\overset{\mathcal{G}_{\alpha}}{\longleftrightarrow} \partial B(f(n))$ means that $B(m)$ and $\partial B(f(n))$ are connected via the hyper-edges in $\mathcal{G}_{\alpha}$.
 \end{lemma}
 The proof of \cite[Lemma~5.5]{ChangMR3692311} does not work and we provide a different approach.
 \begin{proof}
  Denote by $\mathcal{O}(m,n,\alpha)$ the following multi-set of hyper-edges with multiplicities: \[\{h\in\mathcal{G}_{\alpha}: h\subset B(\lfloor\lambda n\rfloor-1), h\cap\partial B(n)\neq\emptyset, h\cap\mathcal{C}(m,n,\alpha)\neq\emptyset\},\]
  where $\mathcal{C}(m,n,\alpha)$ is defined in Lemma~\ref{lem: many truncated hyper-edges}. Conditionally on $\mathcal{C}(m,n,\alpha)$, the number $\sharp\mathcal{O}(m,n,\alpha)$ with multiplicities is a Poisson random variable with parameter $\mu(m,n,\alpha)$. Hence, when $\mu(m,n,\alpha)$ is large, $\sharp\mathcal{O}(m,n,\alpha)$ is also large with high probabilities. Let $\widetilde{\mathcal{O}}(m,n,\alpha)$ be the set obtained from $\mathcal{O}(m,n,\alpha)$ by removing duplicated hyper-edges. We want to show that $\widetilde{\mathcal{O}}(m,n,\alpha)$ is also large with high probabilities. Indeed, this is true since $\sup_{h\in\mathcal{H}}\mu(\{h\})<\infty$ by \eqref{eq: local finiteness of general mu} and the translation-invariance. For $k\geq 1$, when $\mu(m,n,\alpha)$ is large, we can split the hyper-edges $\{h:h\subset B(\lfloor \lambda n\rfloor-1),h\cap\partial B(n)\neq\emptyset,h\cap\mathcal{C}(m,n,\alpha)\neq\emptyset\}$ into $k$ disjoint subsets such that each set has a large enough total weight under $\mu$. Then, with high probabilities, each such set contains at least one hyper-edges in $\mathcal{G}_{\alpha}$, which implies $\sharp\widetilde{\mathcal{O}}(m,n,\alpha)\geq k$. Here, $k$ could be sufficiently large as long as $\mu(m,n,\alpha)$ is sufficiently large. Finally, $\sharp U(m,n,\alpha)$ has to be large if $\sharp\widetilde{\mathcal{O}}(m,n,\alpha)$ is sufficiently large. Hence, we have that
  \begin{equation*}
 \lim\limits_{n\rightarrow\infty}\mathbb{P}[\sharp U(m,n,\alpha)\leq k,B(m)\overset{(\mathcal{G}_{\alpha})^{B(n-1)}}{\longleftrightarrow} \partial B(f(n))]=0,
 \end{equation*}
 which implies \eqref{eq: mtcttob0} by Lemma~\ref{lem: truncated crossing general mu}.
 \end{proof}
 By using similar argument, we obtain the following analogy of \cite[Lemma~5.6]{ChangMR3692311} from Lemma~\ref{lem: many truncated connections to the outer box}.
 \begin{lemma}\label{lem: a connection to a seed}
 If $\alpha>\alpha_c=-\ln(1-u_c)$, for all $\eta>0$, there exist $m=m(d,\alpha,\eta)$ and $n=n(d,\alpha,\eta)>2m$ such that
 \begin{equation*}
  \mathbb{P}[B(m)\overset{(\mathcal{G}_{\alpha})_{B(n-1)}^{B(\lfloor \lambda n\rfloor-1)}}{\longleftrightarrow} K(m,n,\alpha)]>1-\eta,
 \end{equation*}
 where $(\mathcal{G}_{\alpha})_{B(n-1)}^{B(\lfloor \lambda n\rfloor-1)}=\{h\in \mathcal{G}_{\alpha}:h\cap B(n-1)\neq\emptyset, h\subset B(\lfloor \lambda n\rfloor-1)\}$.
 \end{lemma}

 The main difference in the proof of Lemma~\ref{lem: a connection to a seed} from \cite[Lemma~5.6]{ChangMR3692311} is the definition of a seed event. In the loop percolation, if a box $B(x,m)$ is a $\beta$-seed, then each pair of vertices in $B(x,m)$ is connected by the collection of loops $(\mathcal{L}_{\beta})^{B(x,m)}$ inside $B(x,m)$. But for our hyper-edge percolation, since we have no restriction on the shape of hyper-edges, if we consider the collection of hyper-edges $(\mathcal{G}_{\beta})^{B(x,m)}$ inside $B(x,m)$ which is an analogy of $(\mathcal{L}_{\beta})^{B(x,m)}$, this collection of hyper-edges may not connect each pair of vertices in $B(x,m)$. For vertices in the corner of $B(x,m)$, we cannot ensure the connection. Thus, our seed event for the hyper-edge percolation is defined for modified box $B(x,m,c)$ which is the box $B(x,m)$ without $2^d$-many corners. In the proof of \cite[Lemma~5.6]{ChangMR3692311}, a relevant step is the following simple result: Consider certain set of vertices $V(m,n,\alpha)$ of one octant $T(n)$ of a face $F(n)$ of $B(n)$. Assume $2m+1$ divides $n+1$. Then, we may place $(n+1)^{d-1}/(2m+1)^{d-1}$-many disjoint copies of $B(m)$ on $T(n)$ side by side. If $\sharp V(m,n,\alpha)$ is sufficiently large, then we have sufficiently many disjoint copies of $B(m)$ with non-empty intersection with $V(m,n,\alpha)$. The disjointness is important as it ensures the independence between the loop soup inside different boxes. For our hyper-edge percolation, $V(m,n,\alpha)$ is defined analogously. However, we should use modified boxes $B(x,m,c)$ instead of the copies $B(x,m)$ of $B(m)$. We cannot cover $T(n)$ by disjoint faces of the modified boxes $B(x,m,c)$. Frankly speaking, this difference requires additional technical arguments. However, this is not a serious issue. Indeed, by removing the corners of $T(n)$, we obtain $T(n,c)$. Denote by $F(m,c)$ the face $F(m)$ without $2^{d-1}$-many corners, see Figure~\ref{fig: modified face}. Then, $F(m,c)$ is a face of $B(0,m,c)$. Note that $T(n,c)$ could be covered by copies of $F(m,c)$. Moreover, for sufficiently large $m$ and $n$, these copies can be divided into $2^{d-1}$-many groups such that the copies in each group is disjoint. See Figure~\ref{fig: cover without corners} for an illustration of such a covering. Accordingly, there exists a group such that the intersection of $V(m,n,\alpha)$ with the union of the copies in that group is at least $\sharp V(m,n,\alpha)/2^{d-1}$ and the proof follows.

 \begin{figure}[htbp]
 \centering
 \subfloat[]{
 \includegraphics[width=0.136\textwidth]{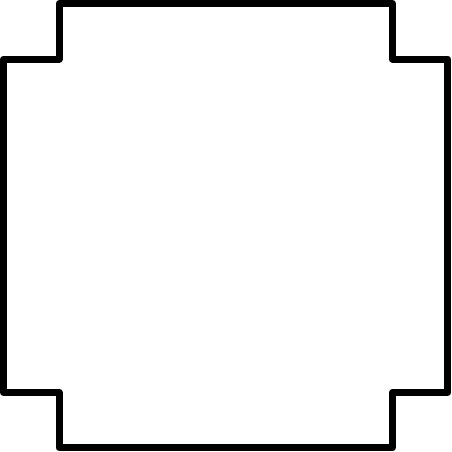}
 \label{fig: modified face}
 }
 \subfloat[]{
 \includegraphics[width=0.5\textwidth]{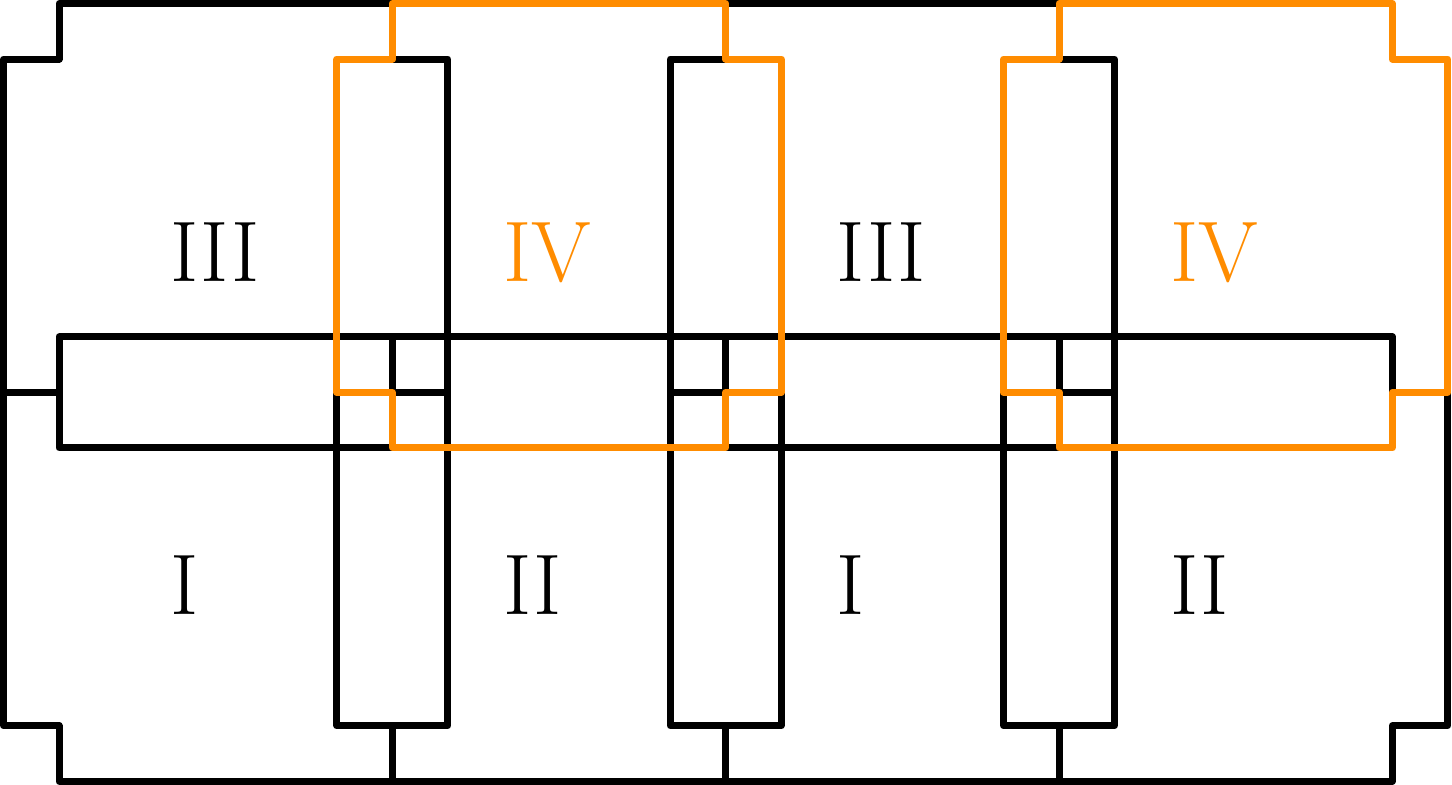}
 \label{fig: cover without corners}}
 \caption{Illustration of the covering, 4 groups}
 \end{figure}

 Finally, by similar argument as in the proof of \cite[Lemma~5.1]{ChangMR3692311}, we could deduce Lemma~\ref{lem: sprinkling for hyper-edge percolation} from Lemma~\ref{lem: a connection to a seed}.

\section{Appendix}

In this section, we prove Claim~\ref{claim: square loop percolates} in Remark~\ref{rem: necessarity on the finiteness of the annulus crossing}. The proof is a simple conclusion of a geometric consideration and the second Borel-Cantelli lemma. We will define a sequence of annulus crossing events on different scales. These events are independent by construction. The sum of their probabilities is infinite and hence, by Borel-Cantelli lemma, these events eventually happens. While by definition, this implies the existence of an infinite chain of square loops and the existence of an infinite cluster.

Without loss of generality, we assume that the underlying graph is $\mathbb{Z}^2$, i.e. $d=2$. Denote by $\mathcal{H}_n$ the set of square loops with side length $2\times 2^{n+1}$ such that their bottom-left corners are placed in the box $\{2^{n-1}+1,2^{n-1}+2,\ldots,2^{n}\}^2$. We illustrate a square loop in $\mathcal{H}_n$ in Figure~\ref{fig: one square loop}. Moreover, $\forall h_1\in\mathcal{H}_n$ and $\forall h_2\in\mathcal{H}_{n+1}$, they are connected, see Figure~\ref{fig: consecutive square loops}. Define
\[E_n=\{\exists h\in\mathcal{H}_n:h\text{ is open}\}.\]
When $\liminf_{n\to\infty}E_n$ happens, there exists an infinite open cluster that is an infinite sequence of connected loops. Note that
\begin{multline*}
 P(E_n^c)=P(\text{all the square loops in }\mathcal{H}_n\text{ are closed})\\
 =\prod_{h\in\mathcal{H}_n}P(h\text{ is closed})=(1-u)^{(n+1)2^{-2(n+1)}2^{2(n-1)}}.
\end{multline*}
Since $\sum_{n}P(E_n^c)<+\infty$ for $u>0$, $P(E_n^c\text{ i.o.})=0$. Hence, $P(\liminf_{n}E_n)=1$ for $u>0$. Therefore, for any $u>0$, with probability one, there is an infinite cluster.

\begin{figure}
\centering
\subfloat[]{
 \includegraphics[width=0.35\textwidth]{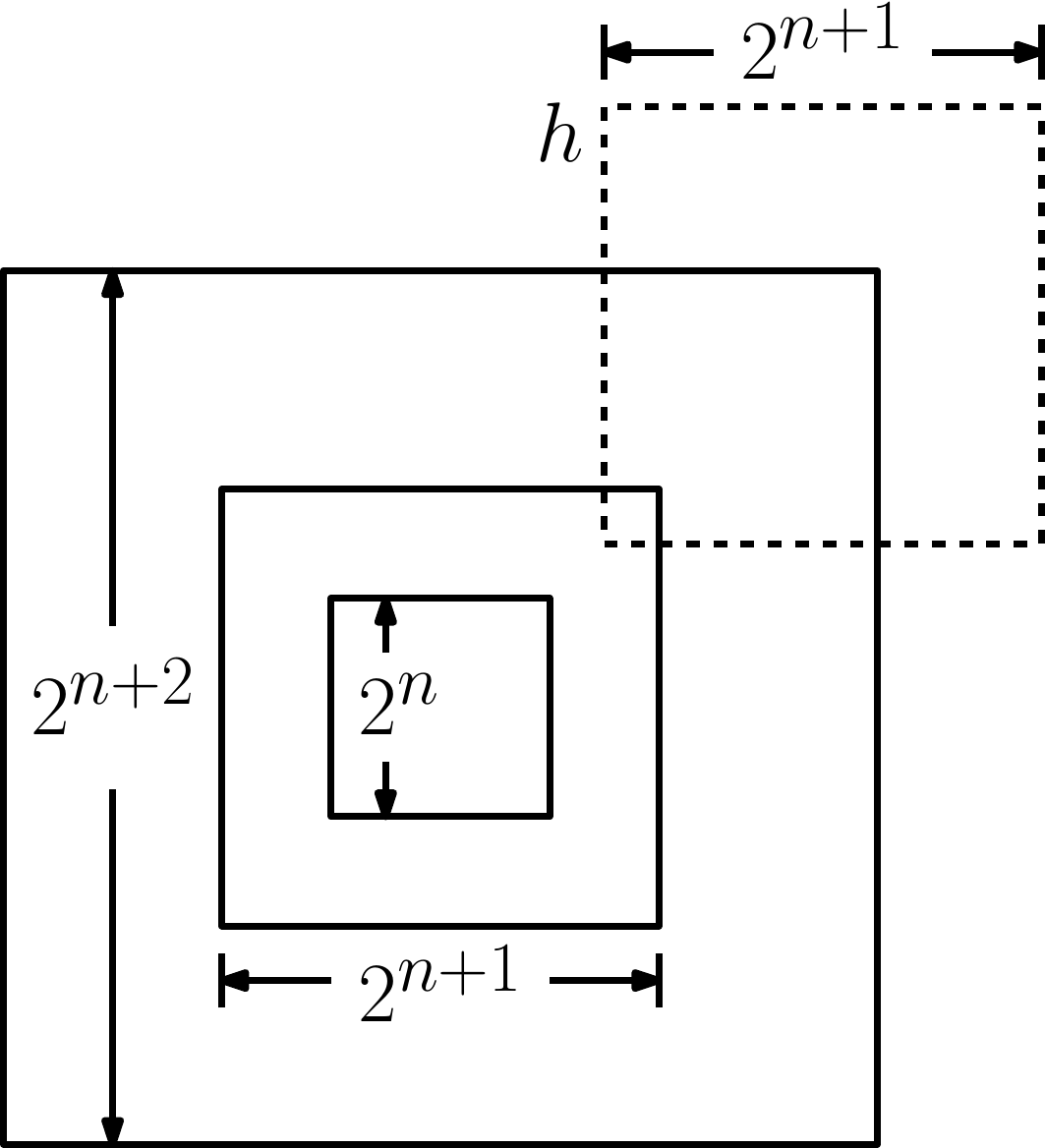}
 \label{fig: one square loop}}
\subfloat[]{
 \includegraphics[width=0.5\textwidth]{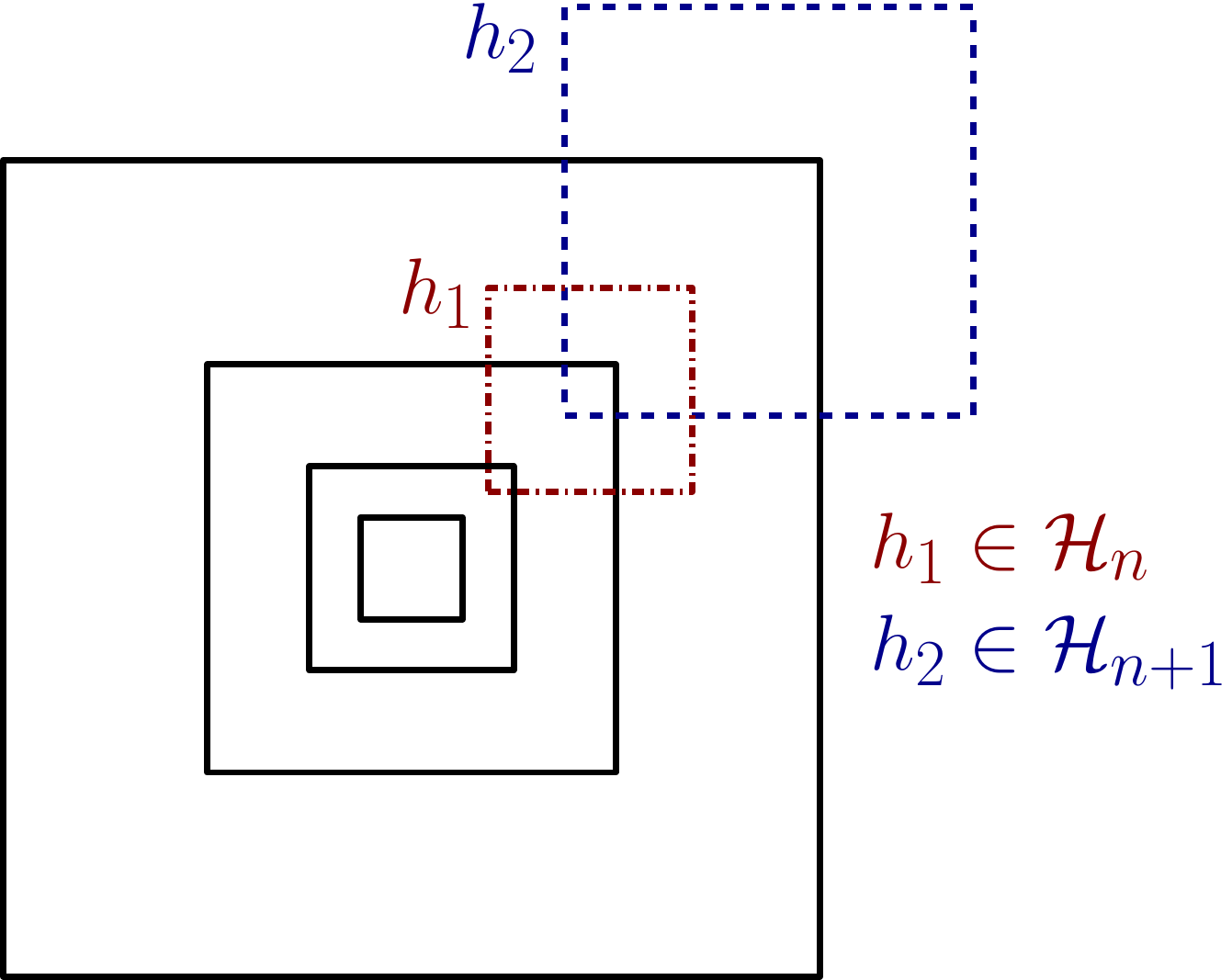}
 \label{fig: consecutive square loops}}
 \caption{Illustration of $\mathcal{H}_n$}
\end{figure}


\providecommand{\bysame}{\leavevmode\hbox to3em{\hrulefill}\thinspace}
\providecommand{\MR}{\relax\ifhmode\unskip\space\fi MR }
\providecommand{\MRhref}[2]{%
  \href{http://www.ams.org/mathscinet-getitem?mr=#1}{#2}
}
\providecommand{\href}[2]{#2}

\end{document}